\documentclass[10pt]{article}
\usepackage{amssymb,amsfonts}
\usepackage{pb-diagram}
\usepackage{amsmath,amsthm,amsfonts,amssymb}
\usepackage[usenames]{color}
\usepackage{hyperref}
\usepackage{soul}
\usepackage{graphicx}
\usepackage{amssymb}
\usepackage{amsmath}
\usepackage{amssymb}

\newtheorem{theorem}{Theorem}

\newtheorem{lemma}{Lemma}
\newtheorem{cor}{Corollary}

\newtheorem{prop}{Proposition}

\newcommand{\be}{\begin{enumerate}}
\newcommand{\ee}{\end{enumerate}}
\newcommand{\beq}{\begin{equation}}
\newcommand{\eeq}{\end{equation}}

\def\N{{\mathbb{N}}}

{}
{}
\title{Undecidability of the first order theories of  free non-commutative  Lie  algebras}
\author{Olga Kharlampovich\footnote{Hunter College, CUNY.}  and Alexei Myasnikov \footnote{Stevens Institute of Technology.}\footnote{The results of this paper were inspired by NSF grant proposal.}}

\date{}

\begin{document}

\maketitle

\begin{abstract} 

Let $R$ be a commutative integral unital domain  and $L$ a free non-commutative Lie algebra over $R$. In this paper we show that  the ring $R$ and its action on $L$   are  0-interpretable in $L$, viewed as a ring with the standard ring language $+, \cdot,0$. Furthermore,  if $R$ has characteristic zero then we prove that  the elementary theory $Th(L)$  of $L$ in the standard ring language is undecidable. To do so   we show  that   the arithmetic $\N = \langle\N, +,\cdot,0 \rangle$ is 0-interpretable in $L$.  This implies  that the  theory of $Th(L)$ has the independence property.  These results answer some old questions on model theory of free Lie algebras.

\end{abstract}

\tableofcontents

\section{Introduction}

In this paper we continue our program on model theory of groups and algebras outlined at the ICM in Korea in 2014 \cite{ICM}. Let $R$ be a commutative integral unital domain  and $L$ a free non-commutative Lie algebra over $R$. We show that  the ring $R$ and its action on $L$   are  0-interpretable in $L$, viewed as a ring in  the standard ring language $+, \cdot,0$.  Furthermore,  if $R$ has characteristic zero then we prove that   the arithmetic $\N = \langle\N, +,\cdot,0 \rangle$ is 0-interpretable in $L$. Hence  the elementary theory $Th(L)$  of $L$ in the standard ring language is undecidable and has  the independence property.  These answer some old questions on model theory of free Lie algebras.  Along the way we further developed  the method that uses maximal rings of scalars in Lie rings that gives a general approach to study first order theories of arbitrary non-commutative  finitely generated  Lie algebras.

The question about decidability of the first-order theory of non-commutative free Lie algebras was well-known in Malcev's school of algebra and logic in Russia. In 1963 Lavrov showed that if the elementary theory $Th(R)$ of  the  integral domain $R$  is undecidable then the elementary theory $Th(L)$ of $L$ is also undecidable. To this end he interpreted the ring $R$ in $L$ \cite{Lavrov}.

 In 1986 Baudisch proved in \cite{Bau} that the theory $Th(L)$ is unstable for every such ring of coefficients $R$. To obtain this result he uniformly  interpreted every initial segment of Presburger arithmetic   in $L$.  Following Lavrov he also showed that the ring $R$ and its action on $L$ are interpretable (with the use of parameters) in $L$. 
 
In the same paper  Baudisch stated the following open problems: Does the theory $Th(L)$ of a  free non-commutative  Lie algebra $L$ over a commutative
integral domain have the independence property? Is $Th(L)$ undecidable? Is it
possible to interpret the initial segments of the natural numbers with addition and
multiplication in it?   Independently, in the book \cite{BK} Bokut' and  Kukin asked a similar   question: for which integral domains $R$ the theory $Th(L)$ is decidable? 

As we have mentioned already, our results completely answer the questions above  in the case when the ring $R$ has characteristic zero.  It seems plausible that similar results hold for arbitrary infinite integral domains $R$. However, our techniques do not work if the ring $R$ is finite, so the following question seems to be very interesting. Is the theory $Th(L)$ undecidable when the ring $R$ is a finite field?  More precisely, is the arithmetic $\N = \langle\N, +,\cdot,0 \rangle$ interpretable in a free non-commutative Lie algebra $L$ over a finite field? 

We would like to  mention that  our proofs seem  general enough to get similar results for some other Lie algebras, in particular, for various  $\N$-graded  Lie algebras where  the maximal rings of scalars are   integral domains.   Actually, we prove that for an arbitrary finitely generated Lie $R$-algebra $L$ over an arbitrary commutative associative unital ring $R$ the maximal ring of scalars of $L$ and its action on  $L/Ann(L)$ and $L^2$ are 0-interpretable in $L$. This gives a general approach to study first order theories of finitely generated  Lie algebras. To interpret arithmetic in such an algebra $L$ one also needs  some weak finiteness divisibility conditions on $L$, which in the case of a free Lie algebra $L$ come from the fact that $L$ is $\N$-graded. 
Note, that the model theory of finite dimensional Lie algebras over fields was studied in \cite{M1}.

This paper is a continuation of the research in \cite{Ass1,Ass2,Ass3} on model theory of free associative algebras. For some time we thought that model theory of free Lie algebras, though very different from the case of  free groups (see \cite{KM1,Sela}),  will be somewhat reminiscent of the model theory of free pro-p-groups (see \cite{MyaR,Lub}). Now, it looks much more like the model theory of free associative algebras, though the proofs are more technical. The main difference is that in free associative algebras a centralizer of a non-invertible element is isomorphic to the ring of polynomials in one-variable, hence the known results from commutative algebra and number theory can be applied. In free Lie algebras we had to exploit some interesting module structures and unusual divisibility arguments.  It seems possible that one can develop current techniques a bit further and study equations  in free Lie algebras as well as elementary equivalence  of such algebras in the way it was done for free associative ones. There are two interesting  open questions here: whether  one can interpret the weak second order theory of the ring $R$ in a free non-commutative Lie algebra $L$ with coefficients in $R$; and if the Diophantine problem in $L$ is decidable.

\section{Maximal rings of scalars}

\subsection{Maximal rings of scalars of bilinear maps}
\label{se:2}

Let $R$  be a commutative associative ring with unity 1, and $M_1, M_2, N$ exact $R$-modules.  Let $f\colon M_1 \times M_2 \to N$ be an $R$-bilinear map.  For a subset $E_1 \subseteq M_1$ we define the  right annulator of $E_1$ (relative to $f$) by  $Ann_r(E_1) = \{y \in M_2 \mid f(E_1,y)=0\} $.  Similarly, for a subset $E_2 \subseteq M_2$ we define the left  annulator of $E_2$ by $Ann_l(E_2) =\{x \in M_1 \mid f(x,E_2) = 0\} $. 

 We say that 
\begin{itemize}
\item [1)] $f$ is  {\em non-degenerate} if $Ann_l(M_2) =0$ and $Ann_r(M_1) = 0$. 
\item [2)] $f$ is {\em onto}   if the submodule (equivalently,  the subgroup) $\langle f(M_1,M_2)\rangle$ generated by $f(M_1,M_2)$ is equal to $N$.
\end{itemize}

Note that the conditions 1) -- 2) do not depend on the ring $R$, i.e., whether they hold or not in $f$ depends only on the abelian group structure of $M$ and $N$.

For any non-degenerate onto bilinear map $f\colon M_1 \times M_2 \to N$ there is a uniquely defined {\em maximal ring of scalars} $P(f)$, which is an analogue of the {\em centroid} of a ring.  More precisely, a commutative associative unital ring $P$ is called a "ring of scalars" of $f$ if $M_1, M_2,$ and $N$ admit the  structure  of faithful   $P$-modules such that $f$ is $P$-bilinear. A ring of scalars $P$ of $f$ is called {\em maximal} if for every ring of scalars $P^\prime$ of $f$ there is a monomorphism $\mu\colon P^\prime \to P$ such that for every $\alpha \in P^\prime$ its actions on $M_1,M_2,$ and $N$ are the same as the actions of $\mu(\alpha)$. It is easy to see  that the maximal ring of scalars of $f$  exits, it is unique up to isomorphism, as well as its actions on $M_1, M_2,$ and $N$. We denote the unique maximal ring of scalars of $f$  by $P(f)$. In fact, the ring $P(f)$ can be constructed as follows. 

Let $End(M_1), End(M_2), End(N)$ be the ring of endomorphisms of $M_1,M_2,$ and $N$  (here $M_1,M_2,$ and $N$ are  viewed as  abelian groups). Below for an endomorphism $\beta$ and an element $x$ the image of $\beta$ on $x$ is denoted by $\beta x$.

 If $P$ is a ring of scalars of $f$ then the actions of $P$ on $M_1,M_2,$ and $N$ give embeddings $P \to End(M_i)$, $P \to End(N)$, $i = 1,2$,  which give rise to the diagonal embedding $\Phi: P \to  End(M_1) \times End(M_2) \times End(N)$. Denote the direct product of rings  $End(M_1) \times End(M_2) \times End(N)$ by $K(f)$. Let 
$\tau_i : K(f) \to M_i$, $\sigma:K(f) \to N$ be the canonical projections of $K(f)$ onto its direct factors.
Since $P$ is a ring of scalars of $f$  every $\alpha \in \Phi(P) \leq K(f)$ satisfies the following conditions for any $x \in M_1, y \in M_2$:
\begin{equation} \label{eq:K(f)}
f(\tau_1(\alpha)x,y)= f(x,\tau_2(\alpha)y) = \sigma(\alpha)f(x,y).
\end{equation}
It is not hard to see that the set $P(f)$ of all elements $\alpha \in K(f)$ which satisfy the condition (\ref{eq:K(f)}) is a commutative unital subring of $K(f)$. We showed above that every ring of scalars of $f$ embeds into $P(f)$ in such a way that its action on $M_1,M_2,N$ agrees with the action of $P(f)$. Hence $P(f)$ is the maximal ring of scalars of $f$. 

To interpret $P(f)$ in $f$ we need  another  description of $P(f)$.  Let $M(f) = End(M_1) \times End(M_2)$  and $\tau = \tau_1 \times \tau_2$ be the canonical projection of $K(f)$ onto $M(f)$. As we mentioned above we may assume that $P(f)$ is a subring of $K(f)$, the restriction of $\tau$ on $P(f)$ gives   a homomorphism $\tau: P(f) \to M(f)$. Clearly, $\tau: P(f) \to M(f)$ is injective and for every  $\alpha \in \tau(P(f))$  the following conditions (S) and ($W_n$) hold for every  $n \in \mathbb{N}$:

\begin{itemize}
\item [(S)] for every $x \in M_1, y \in M_2$ 
$$ f(\tau_1(\alpha)x,y)= f(x,\tau_2(\alpha)y);$$
\item [($W_n$)] for every $x_{k},x_{k}' \in M_1, y_{k},y_{k}' \in M_2$, $ k = 1, \ldots,n$
$$\Sigma_{k = 1}^n f(x_k,y_k) = \Sigma_{k = 1}^n f(x_k',y_k')  \rightarrow \Sigma_{k = 1}^n f(\tau_1(\alpha)x_k,y_k) = \Sigma_{k = 1}^n f(\tau_1(\alpha)x_k',y_k') $$

\end{itemize}
We claim that $\tau(P(f))$ consists precisely of those elements $\alpha \in M(f)$ for which the conditions (S) and $(W_n)$ hold for every $n \in \N$.
 Denote by $Sym(f)$ the subset of all ($f$-symmetric) elements $\alpha \in M(f)$ which satisfy (S) and by $W_n(Sym(f))$ the subset of those $\alpha \in Sym(f)$ which satisfy ($W_n$). 
 Put 
 $$P_{SW}(f)  = \bigcap_{n = 1}^\infty W_n(Sym(f)).$$
Clearly, $\tau(P(f)) \subseteq P_{SW}(f)$. To show the equality it suffices to show that for every $\alpha \in P_{SW}(f)$ there is $\sigma \in End(N)$ such that (1) holds, i.e., for any $x \in M_1, y \in M_2$ 
$$f(\tau_1(\alpha)x,y)= f(x,\tau_2(\alpha)y) = \sigma f(x,y).$$
To this end for a given $\alpha \in P_{SW}(f)$ and given $x \in M_1, y \in M_2$ define $\sigma f(x,y) = f(\tau_1(\alpha)x,y)$. Since $\alpha$ satisfies $(W_1)$ this definition is correct, i.e., for any $x' \in M-1, y' \in M_2$ one has $\sigma f(x,y) = \sigma f(x',y')$.  Similarly, since $\alpha$ satisfies all the conditions $(W_n)$ one can correctly extend the definition of  $\sigma$ by linearity   on the whole subgroup $N_0$ generated in $N$ by the set $f(M_1,M_2)$.  Since $f$ is onto $N_ 0 = N$,  so    $\sigma \in End(N)$ and (1) holds, as required. This shows that $\tau(P(f)) =  P_{SW}(f)$, as claimed.

To study model theoretic properties  of $f\colon M_1 \times M_2 \to N$ one associates with $f$ a three-sorted structure ${\mathcal A}(f) = \langle M_1, M_2, N; f \rangle$, where $M_1, M_2,$ and $N$ are abelian groups equipped with the map  $f$ (the language of ${\mathcal A} (f)$ consists of additive group languages for $M_1, M_2,$ and $N$, and the predicate symbol for the graph of $f$). Our goal is to show that the ring $P(f)$ as well as its actions on the modules $M_1,M_2$ and $N$, are interpretable in the structure ${\mathcal A}(f)$. 
For this we need $f$ to satisfy some finiteness conditions.

 We say that 
\begin{itemize}
\item [3)]   a finite subset $E_1 \subseteq M_1$ is called a  {\em  left complete system} for $f$ if  $Ann_r(E_1) = Ann_r(M_1)$. Similarly,    a finite subset $E_2 \subseteq M_2$ is called a  {\em  right complete system} for $f$  if  $Ann_l(E_2) = Ann_l(M_2)$. In this case we say that a pair $(E_1,E_2)$ is a finite complete system for $f$.
\item [4)] $f$ has {\em finite width }  if   there exists some natural number $m$, such that for any $z \in N$ there are some $x_i\in M_1, y_i \in M_2, i = 1, \ldots,m$ such that $z=\sum _{i=1}^mf(x_i,y_i)$. The least such  $m$ is termed the width of $f$.
\end{itemize}

\begin{theorem} \label{th:bilinear} \cite{M} Let $f$ be an $R$-bilinear map $M_1\times M_2\rightarrow N$ that satisfies 1)-4)  above. Then the maximal ring of scalars  $P(f)$ for $f$ and its actions on $M$ and $N$ are 0-interpretable  in ${\mathcal A}(f)$ uniformly in the size of the finite complete system and the width of $f$.
\end{theorem}

\subsection{Maximal rings of scalars of finitely generated Lie algebras}

In this section we prove some results on  maximal rings of scalars in finitely generated Lie algebras and also in free Lie algebras of arbitrary rank. 

Assume that $R$ is an integral domain (commutative associative and unital). Let $L$ be a Lie $R$-algebra. Denote by $L^2$ the $R$-submodule of $L$ generated by all products $xy$ where $x, y \in L$. Then the multiplication map   $f_L\colon L\times L\rightarrow L^2$ is $R$-bilinear and onto.  This map induces a non-degenerate $R$-bilinear onto map ${\bar f}_L\colon L/Ann (L)  \times L/Ann( L) \to L^2$, where $Ann(L)  = \{x \in L \mid xL = 0\}. $

\begin{lemma} \label{le:1-4} Let $L$ be a finitely generated  Lie $R$-algebra. Then the bilinear map ${\bar f}_L$ satisfies all the conditions 1)--4). In particular, if $Ann(L) = 0$ then the multiplication $f_L$ satisfies all the conditions 1)--4).
\end{lemma}
\begin{proof}
Suppose $L$ is generated (as an algebra) by a finite set $X$.  The map ${\bar f}_L$ satisfies conditions 1) and 2) by construction. To prove 3) it suffices to show that $Ann(L) = Ann(X)$. Let $a \in Ann(X)$ and $b \in L$. To show that $ab = 0$ we may assume by linearity that $b$ is a product of elements from $X$. If $b \in X$ then $ab = 0 $, otherwise, $b = uv$, where $u, v$ are products of elements of $X$ of shorter length. By induction on length $au = av = 0$. Since  $L$ is Lie then $a(uv) = -u(va) -v(au) = u(av)-v(au) = 0$, hence the claim. To show 4)  we  prove that $L^2 = Lx_1 + \ldots + Lx_n$, where $X = \{x_1, \ldots, x_n\}$. Clearly, it suffice to show that every product $p$ of elements from $X$ belongs to $M =  Lx_1 + \ldots + Lx_n$.   Note that  $p = uv$ for some Lie words $u, v$ in $X$. We use induction on the length of $v$ (as a Lie word in $X$) to show that $p \in M$. If $ v$ is an element from $X$ then there is nothing to prove.  Otherwise, $v = v_1v_2$ where $v_1, v_2$ are Lie words in $X$ of smaller length. Then $u(v_1v_2) = - v_1(v_2u) -v_2(uv_1) = (v_2u)v_1 + (uv_1)v_2$. Now by induction on the length of the second factors we get that 
$(v_2u)v_1, (uv_1)v_2$, and hence $(v_2u)v_1 + (uv_1)v_2$, are in $M$, as required.
\end{proof}

\begin{theorem} \label{th:interpret-Lie}
Let $L$ be a finitely generated  Lie $R$-algebra. Then the maximal ring of scalars of the bilinear map ${\bar f}_L$ and its action on $L/Ann(L)$ and $L^2$ are 0-interpretable in $L$ (viewed in the language of rings) uniformly in the size of a finite generating set of $L$.
\end{theorem}
\begin{proof}
Let $A$ be a finite generating set of $L$. As was shown in Lemma \ref{le:1-4} the set $L^2$ is 0-definable in $L$ uniformly in the size of the set $A$. Hence the bilinear map ${\bar f}_L$, i.e., the structure $\mathcal{A}(\bar f_L)$, is 0-interpretable in $L$ uniformly in the size of $A$. Now by Theorem \ref{th:bilinear} the maximal ring of scalars of ${\bar f}_L$  and its action on $L/Ann(L)$ and $L^2$ are 0-interpretable in $\mathcal{A}({\bar f}_L)$, hence in $L$, uniformly in the size of a finite complete system of ${\bar f}_L$ and the width of ${\bar f}_L$, which by  Lemma \ref{le:1-4} are uniform in the size of  $A$. This proves the theorem.
\end{proof}

\subsection{Maximal rings of scalars of free  Lie algebras}

Let $L$ be a free Lie algebra with  finite set of free generators  $X$ over an integral domain $R$. 

An element $u \in L$ can be uniquely decomposed as a sum of homogeneous elements $u = u_1 + \ldots u_n$  of pair-wise distinct weights (or degrees) with respect to system of free generators $X$. Notice that $u = 0 \longleftrightarrow u_1 =0, \ldots, u_n = 0$. 
By $\bar u$ we denote the homogeneous component of $u$ of the highest weight.  By $wt(u)$ we denote the weight of $\bar u$. Observe, that $wt(\bar u \bar v) = wt(\bar u)+wt(\bar v)$ provided $\bar u \bar v \neq 0$. 

Denote by  $\mathcal H$ the set of  Hall basis commutators on $X$ (see  \cite{Magnus} or \cite{Bahturin}), then $\mathcal H$ forms an $R$-basis of $L$ as the $R$-module. We need the following well-known result, furthermore, since we need the argument used in its proof we provide a short proof as well. 

\begin{lemma} \label{le:centralizer}
Let $L$ be a free non-commutative Lie algebra with  system of free generators  $X$ over an integral domain $R$. Then:
\begin{itemize}
\item [1)] for any non-zero  $u,v \in L$ if $ u v = 0$ then $\alpha u = \beta v$ for some non-zero  $\alpha, \beta \in R$.
\item [2)] Let $u \in \mathcal H$ be a basic commutator over $X$. Then for any $v \in L$ if $uv = 0$ then there is $\alpha \in R$ such that $v = \alpha u$.
\end{itemize}
\end{lemma}
 \begin{proof} 
 To show 1) let $u,v \in L$ and $u = \sum u_i, v = \sum v_j$ be their decompositions on homogeneous components. Assume that $u_1 = \bar u, v_1 = \bar v$.  Since $uv = 0$ it follows that  $\bar u \bar v = 0$. Then by Theorem 5.10 from  \cite{Magnus} $\alpha \bar u = \beta \bar  v$ for some $\alpha, \beta \in R$. Consider $u' = \alpha u - \beta v$ then $u'v = 0$ and $wt(u') < wt(u)$. The argument above shows that the components of the highest weight in $u'$ and $v$ are linearly dependent, hence either of the same weight,  or $u' = 0$. Since $wt(u') < wt(v)$ we get $u' = \alpha u - \beta v = 0$, as claimed.  
 
 To prove 2) take $u \in \mathcal H$. Suppose $uv = 0$ for some $v \in L$. Consider the decomposition $v = \sum_i v_i$ of $v$ into homogeneous components with respect to $X$. Then $uv = \sum_iuv_i= 0$ hence $uv_i= 0$ for each such $i$. It follows from  Theorem 5.10 in \cite{Magnus} that $u$ and $v_i$ are linearly dependent over $R$. Since $\mathcal H$ is an $R$-basis of $R$ it follows that $v$ is homogeneous of the same weight as $u$ and $\alpha v  = \beta u$ for some $\alpha,\beta \in R$. Since $v$ is in the same homogeneous component as $u$ it follows that $v = \sum_i \alpha_iu_i$ where $u_i$  are the basic commutators from $\mathcal H$ of the same weight as $u$, so $u$ is one of them, say $u = u_1$.  The equality  $\alpha v  = \sum_i \alpha \alpha_iu_i = \beta u_1$  implies  that $\alpha_i =0 $ for $i \geq 2$ and $\alpha\alpha_1 = \beta$. Hence  $\alpha v  = \alpha\alpha_1 u$, so $v = \alpha_1 u$, as claimed.
 \end{proof}

\begin{prop}\label{th:scalar} Let $L$ be a  non-commutative free Lie algebra over an integral domain  $R$.  Then the maximal ring of scalars $P(f_L)$ of the multiplication bilinear map $f_L$ is isomorphic  to the ring  $R$.
\end{prop}

\begin{proof}  

Let $L$ be a free Lie algebra over $R$ with system of free generators $X$. Notice first that $Ann_l(L) = Ann_r(L) = 0$ and $f_L$ is onto (see Lemma \ref{le:1-4}), so the maximal ring of scalars $P = P(f_L)$ exists.

 Let $\mathcal H$ be a Hall basis of $L$. By Lemma \ref {le:centralizer}   for any $x \in \mathcal H$ and $a  \in L$ if $ax =0$ then $a \in Rx$. Let $\alpha \in P$, then the action of $\alpha$ on $L$ gives an $R$-endomorphism $\phi_\alpha$ of $R$-module $L$ such that $\phi_\alpha(xy) = \phi_\alpha(x)y = x\phi_\alpha(y)$. Hence the action by $\alpha$ is completely determined by its action on $\mathcal H$. Take an arbitrary $x \in \mathcal H$.  One has, $\phi_\alpha(xx) = 0 = (\phi_\alpha(x)x)$, so $\phi_\alpha(x) \in Rx$, say $\phi_\alpha(x) = \alpha_xx$, where $\alpha_x \in R$. Similarly, for $y \in \mathcal H$ $\phi_\alpha(y) = \alpha_yy$ for some $\alpha_y \in R$. It follows that $\phi_\alpha(xy) = \alpha_x (xy) = \alpha_y (xy)$, hence $\alpha_x = \alpha_y$ for any $x, y \in \mathcal H$. Therefore, $\phi_\alpha$ acts on $L$ precisely by multiplication of $\alpha_x$. This shows that $P = R$. 

\end{proof}

 From Theorem \ref{th:bilinear} and Proposition \ref{th:scalar} we get the following result.
 
\begin{cor} \label{th:field} Let $L$ be  a non-commutative  free Lie algebra of finite rank   over an integral domain $R$.  Then the ring $R$ and its action on $L$ is 0-interpretable in $L$ uniformly in the rank of $L$. 
\end{cor}

Notice that Theorem \ref{th:interpret-Lie} gives the result for any finitely generated non-commutative free Lie algebra. To get an interpretation of $R$ and its action on $L$ for an arbitrary non-commutative Lie algebra over $R$ one needs to work a bit more. 

\begin{theorem}
 Let $L$ be  a non-commutative  free Lie algebra  over an integral domain $R$. Then the ring $R$ and its action on $L$ are 0-interpretable in $L$ uniformly on the class of such algebras $L$.
\end{theorem}
\begin{proof}
Before going into details we outline the scheme of the proof first.

For an element $x \in L$ denote by $C(x)$ the centralizer of $x$ in $L$, i.e., $C(x) = \{z \in L \mid xz = 0\}$.
Then for any $x, y \in L$ such that $xy \neq 0$ multiplication in $L$ gives an $R$-bilinear map $F_{x,y}: C(x) \times C(y) \to C(xy)$, which is non-degenerate.  Then there is a maximal ring of scalars $P_{x,y} = P(f_{x,y})$ of $f_{x,y}$.  If $y$ is a basic commutator in $L$ (with respect to some fixed free set of generators $A$ of $L$) then 
by Lemma \ref{le:centralizer} $C(y) = Ry$, so $P_{x,y} = R$. 
Obviously, $f_{x,y}$ (i.e., the structure $\mathcal{A} (f_{x,y})$) is interpreted in $L$ with parameters $x, y$. Observe that $x, y$ form a complete system for $f_{x,y}$.  Hence the group $Sym(f_{x,y})$ is interpreted in $\mathcal{A} (f_{x,y})$. Since $P_{x,y} = R$ all elements from $Sym(f_{x,y})$ satisfy the conditions $W_n$ above, so $P_{x,y} = Sym(f_{x,y})$, hence as we mentioned above the ring $P_{x,y} = R$ is 0-interpreted in $\mathcal{A} (f_{x,y})$, hence in $L$  (with parameters $x,y$). Furthermore, it gives an interpretation of the  action of $R = P_{x,y}$ on $C(x)$ and $C(y)$. If $z$ is  another non-zero element in $L$ then the map $f_{x,z}$ gives another interpretation of $R$ in $L$ as $P_{x,z}$, and also another  interpretation of its action on $C(x)$ and $C(z)$. Comparing the action of $P_{x,y}$ and $P_{x,z}$ on $C(x)$ one can define by formulas of $L$ an isomorphism $P_{x,y} \to P_{x,z}$ uniformly in the parameters $x,y,z$. Identifying elements in $P_{x,y}$ and $P_{x,z}$ along the isomorphism $P_{x,y} \to P_{x,z}$  one can get 0-interpretation of $R$ in $L$ and its action on $L$. 

Since $L$ is a free Lie algebra everything is easier then in arbitrary finitely generated Lie algebras, so one can follow the strategy outlined above and get down to the  precise formulas that 0-interpret $R$ in $L$ and its action on $L$ as follows.  

Let $x \in L$, $x \neq 0$.  The formula
$$
\phi(x,z) = (x \in C(z)) \wedge \forall e \exists e' \in C(e) (xe = ze')
$$
defines in $L$ the predicate $x \in Rz$ (here by $x \in C(z)$ we denote the formula $xz = 0$). Indeed, if $x = \alpha z$ then $xz = 0$. Take an arbitrary  $e \in L$ and put $e' = \alpha e$. Then  $e' \in C(e)$ and  $xe = \alpha ze = z \alpha e = ze'$, as required.  Conversely, suppose  $\phi(x,z)$ holds in $L$ on $x,z$. Then take a basic commutator $e \in L$ that does not appear in the decomposition of $x$ and $y$ into non-trivial linear combinations of basic commutators in $A$.  Since $e' \in C(e)$ it follows from Lemma Lemma \ref{le:centralizer} that $e' = \alpha e$ for  some $\alpha \in R$. The equality $xe = ze' = z \alpha e$ implies that $(x-\alpha z)e = 0$, so $x-\alpha z \in C(e)$. Because of the choice of $e$ the latter can happen only if 
$x-\alpha z = 0$, i.e., $x \in C(z)$, as claimed.  

Recall that elements of $Sym(f_{x,y})$ are interpreted in $f_{x,y}$  by the values on the complete system $x,y$, i.e.,  as elements $(rx,ry), r \in R$. This  gives the following interpretation.  For a fixed $0\neq x \in L$ we turn $Rx$ into a ring by interpreting  an addition $\oplus$ and a multiplication $\otimes$   as follows. We put  $xr \oplus xs$ as the standard addition in $L$, so $xr \oplus xs = xr+xs = x(r+s)$.   To define the multiplication $\otimes$ we need to interpret first the following  predicate on $x,x',y,y' \in L$:
$$
 \exists r \in R (x' = rx \wedge y' = ry).
$$
It is easy to see that the condition above holds on elements $x,x',y,y' \in L$ if these elements satsify the following formula 
$$
\Phi(x,x',y,y') = (x' \in Rx) \wedge (y' \in Ry) \wedge (x'y = xy').
$$  
Now we  define the  multiplication $\otimes$ on $Rx$: if $x_1, x_2,x_3 \in Rx$ then 
$$
x_1 \otimes x_2 = x_3 \Longleftrightarrow \forall y \neq 0  \exists y' \in L \exists s,t \in R (x_2 = sx \wedge y' = sy \wedge x_3 = tx \wedge x_1y' = txy).
$$
The condition on the right can be written by a formula in the ring language using the formula $\Phi(x,x',y,y')$ above.  Observe that the multiplication $\otimes$ corresponds to the multiplication in $R$. Indeed, since $x_1, x_2,x_3 \in Rx$ then $x_1 = rx, x_2 = sx, x_3 = tx$ for some $r,s,t \in R$. For any $0\neq y \in L$ there is $y' = sy$, hence  $x_1y' = rs(xy)$, and then $x_3 = rsx$, as required. 

The argument above shows that we interpreted the ring $R$ as the structure $R_x = \langle Rx : \oplus_x, \otimes_x \rangle$ in $L$ with the parameter $x \neq 0$ uniformly in $x$.  The formula $\Phi(x,x',y,y')$ defines  an isomorphism $R_x \to R_y$ which maps $x' \to y'$. Indeed, if $\Phi(x,x',y,y')$ holds in $L$ on elements $x,x',y,y'$ then $x' = rx, y' = ry$ for some (unique) $r \in R$. Thus, for each non-zero $x,y \in L$  we defined an isomorphism $R_x \to R_y$ uniformly in $x,y$.  Now consider a definable subset in $L\times L$: 
$$
D = \{(x',x) \mid x\neq 0, x' \in Rx\}.
$$
The formula $\Phi(x,x',y,y')$ defines an equivalence relation $\sim$ on $D$. Moreover, the formulas that  interpret operations $\oplus_x$ and $\otimes_x$ on $R_x$ uniformly in $x \neq 0$ allow one to define by formulas operations $\oplus$ and $\otimes$ on the set of equivalence classes $D/\sim$. Indeed, for $\ast_x \in \{\oplus_x, \otimes_x\}$ for $(x_1',x_1), (x_2',x_2), (x_3',x_3) \in D$ put 
$$
[(x_1',x_1)] \ast [(x_2',x_2)] =  [(x_3',x_3)] \Longleftrightarrow  \exists z_1, z_2, z_3, z [ z_1 \ast_z z_2 = z_3 \bigwedge_{i = 1}^3 (z_i,z) \sim (x_i',x_i)]
$$

These define operations $\oplus$ and $\otimes$ on $D/\sim$ such that the resulting structure $R_D  = \langle D/\sim : \oplus, \otimes \rangle$  is isomorphic to $R$. Notice that this interpretation does not use any parameters from $L$.  The formula $\Phi(x,x',y,y')$ defines an action of an element $[(x',x)] \in R_D$ on an arbitrary non-zero element $y \in L$, where the result of this action is an element $y' \in L$ such that $(x',x) \sim (y',y)$. 

This proves the theorem.

\end{proof}

\subsection{Definability of the rank}

Now we show that the rank of a free Lie algebra is definable by first-order formulas.

Recall that a Lie ring $L$ has $L^2$ of {\em finite width} if there is a number $m$ such that every element $w \in L^2$ is equal to a sum of the type $u_1v_1 + \ldots +u_mv_m$ for some $u_i, v_i \in L$. The minimal such $m$ is called the {\em width} of $L$. 

We showed in the proof of Lemma \ref{le:1-4} that every finitely generated Lie algebra has finite width.

\begin{lemma} \label{le:width}
Let $L$ be a  Lie algebra. Then:

\begin{itemize}
\item [1)] the sentence 
$$
 \forall u_1,v_1, \ldots,u_{m+1}, v_{m+1} \exists u_1',v_1', \ldots,u_m',v_m' (\sum_{i = 1}^{m+1} u_iv_i = \sum_{j= 1}^mu_j'v_j')
$$
holds in $L$ if and only if the width of $L^2$  is finite and is less or equal to $m$.

\item [2)]  Consider a formula  
$$
\psi_m(a_1, \ldots,a_m) =   \forall u_1,v_1, \ldots,u_{m+1}, v_{m+1} \exists ,v_1', \ldots,v_m' (\sum_{i = 1}^{m+1} u_iv_i = \sum_{j= 1}^ma_jv_j').
$$
Then if  $L$  is generated as an algebra by elements $u_1, \ldots,u_m$ then $\psi_m(u_1, \ldots,u_m)$ holds in $L$. Furthermore, if $\psi_m(a_1, \ldots,a_m)$ holds in an arbitrary algebra Lie $L$ on some elements   then $L^2$ is of  width at most $m$ in $L$ and it is defined in $L$ by the following formula
$$
S_m(y) = \exists a_1, \ldots,a_m \exists v_1, \ldots, v_m (\psi_m(a_1, \ldots,a_m) \wedge y = \sum_{i+1}^m a_iv_i).
$$
\end{itemize} 
\end{lemma}
 \begin{proof}
 By a straightforward argument.
 \end{proof}
 
 \begin{cor}
 Let $R$ be an integral domain and $L$ a free Lie $R$- algebra of finite rank.  Consider the following formula:
 $$ \label{eq:basis}
 \phi_m(a_1, \ldots,a_m) =  \forall y  \exists \alpha_1, \ldots, \alpha_m \in R \exists z_1, \ldots,z_m \in L(y = \Sigma_{i = 1}^n\alpha_ia_i + \Sigma_{i = 1}^n a_iz_i),
 $$
 where $\alpha_i \in R$ and $\alpha_ia_i$ mean the corresponding formulas from the interpretation of $R$ and its action on $L$ from Theorem \ref{th:field}. Then:
 \begin{itemize}
 \item [1)] the formula 
 $$\Delta_m = \exists a_1, \ldots,a_m (\phi_m \wedge \psi_m)$$
 (here $\psi_m$ is the formula from Lemma \ref{le:width}) holds  in $L$
  if and only if the rank of $L$ is at most $m$.
 \item [2)]  the formula $\Delta_m \wedge \neg \Delta_{m-1}$ holds in $L$ if and only if $L$ has rank $m$.
  \end{itemize}
 \end{cor}
\begin{proof}
To see 1) suppose that $\Delta_m$ holds in $L$, so there are elements $u_1, \ldots,u_m \in L$ such that $\phi_m$ and $\psi_m$ both hold on $u_1, \ldots,u_m $. Then  $L/L^2$ as an $R$-module is generated by $m$ elements.  Conversely, suppose the rank of $L$ is  at most $m$. Then there are elements $u_1, \ldots,u_m \in L$ that generate $L$. Hence by Lemma \ref{le:width}  $\psi_m$ holds in $L$ on $u_1, \ldots,u_m $ and  $L^2 = Lu_1 + \ldots +Lu_n$. Note also that $L$ is generated modulo $L^2$ by $u_1, \ldots,u_m$ as an $R$-module, so the formula $\Delta_m$  holds in $L$.   This proves 1) and 2) now follows from 1). 

\end{proof}

\section{Interpretability of the arithmetic } \label{se:3}

Let $A=\{a,b,  a_1,\ldots, a_n \}$ be a system of free generators  of a free Lie algebra $L$ with coefficients in an integral domain $R$.

By  $(z_1,z_2,\ldots ,z_n)$  we denote the left-normed product of elements $z_1,z_2,\ldots ,z_n$ in $L$. 
For $u,v \in L$ and $\alpha \in R$ by $u(v+\alpha)$ we denote the element $uv +\alpha u \in L$ and refer to it as a "product" of $u$ and $v+\alpha$.

Now we establish some properties of the action above:

\begin{itemize}

\item [a)] for any $u, v, w \in L$  and any $\alpha \in R$ 
$$(u+w)(v+\alpha) = u(v+\alpha) + w(v+\alpha)$$
This  is obvious
\item [b)] for any $u,v  \in L, \alpha, \beta \in R$ one has
$$
(u(v+\alpha))(v+\beta) = (u(v+\beta))(v+\alpha) 
$$
This comes from straightforward verification. Because of this we will omit parentheses in such situations and simple write $u(v+\alpha)(v+\beta)$.

\item [c)]  For any $u,v \in L$ and any $\alpha, \beta  \in R$ the following holds:
$$
\beta u(v +\alpha) = u(\beta v + \beta \alpha), \ \ \ u(v+\alpha) = u((v+\beta u) +\alpha)
$$

\item [d)] For any $u,v  \in L$ and any $\alpha \in R$ one has
$$u(v+\alpha) = 0 \longleftrightarrow u = 0$$
Indeed, if $uv +\alpha u = 0$ then ${\bar u}{\bar v} = 0$, hence by Lemma \ref{le:centralizer} either $\bar v = 0$,  or  $\bar u = 0$, or $r{\bar u} = s{\bar v}$ for some non-zero $r, s \in R$. If $\bar u = 0$ then $u = 0$, as claimed. If $\bar v =0$ then $v = 0$ hence $0 = u(v+\alpha) = uv+\alpha u = \alpha u$, so $u = 0$. 
Suppose now that $r{\bar u} = s{\bar v}$ for some non-zero $r, s \in R$.  Put $v' = sv-ru$. Then by c) 
$$
u(v'+s\alpha) = s(u(v+\alpha)) = 0
$$
and ${\bar u}{\bar v} \neq 0$ unless $v' = 0$. The argument above shows that $v' = 0$, but then as was mentioned above $u = 0$, as claimed.

\item [e)] For any $u,v \in L$  and any $\alpha_1, \ldots, \alpha_n \in R$ if  $uv  \neq 0$ and $\bar u \bar v \neq 0$ then 
$$
wt(u (v+\alpha_1)\ldots (v+\alpha_n))= wt(u) +nwt(v).
$$
This property follows by induction on $n$. In general the following holds:
\item [f)] For any $u,v \in L$  and any $\alpha_1, \ldots, \alpha_n \in R$ if  $uv  \neq 0$  then 
$$
wt(u (v+\alpha_1)\ldots (v+\alpha_n))= wt(u) +nwt(v').
$$
where $v' =v$ if  $\bar u \bar v \neq 0$, otherwise $v' = \beta v - \alpha u$, where $\alpha ,\beta \in R\smallsetminus \{0\}$ are such that   $\alpha \bar u =  \beta \bar v$ (such  $\alpha ,\beta \in R$ always exist if  $\bar u \bar v = 0$).

Indeed, suppose $uv \neq 0$ but $\bar u \bar v = 0$. Fix any   $\alpha ,\beta \in R\smallsetminus \{0\}$  such that   $\alpha \bar u =  \beta \bar v$. Put $v' = \alpha u -\beta v$. Notice that $uv' \neq 0$ and also $wt(v') < wt(v) = wt(u)$ so  $\bar u \bar v' \neq 0$. Denote
$$
w = u (v+\alpha_1)\ldots (v+\alpha_n).
$$
Then by c)  
$$
\beta^n w = u(\beta v +\beta \alpha_1)\ldots (\beta v +\beta \alpha_n) = u(v'+\beta \alpha_1) \ldots (v'+\beta \alpha_n).
$$
Notice that  $wt(w) = wt(\beta^n w)$. It follows from e) that 
$$
wt(\beta^n w) = wt(u) +nwt(v'),
$$
as claimed.

\end{itemize}

The following result holds in any Lie $R$-algebra.
\begin{lemma} \label{le:main}
Let $L$ be any Lie $R$-algebra.  If $ u, v \in L$ and $\alpha_1, \ldots, \alpha_n$ are pair-wise distinct elements from $R$ such that 
$$
u = u_1(v+\alpha_1), \ldots, u = u_n(v+\alpha_n),
$$
for some elements $u_1, \ldots, u_n \in L$ then 
$$
\gamma u = w(v+\alpha_1) \ldots (v+\alpha_n)
$$
for some element $w \in L$ and $0 \neq \gamma \in R$ .
\end{lemma}
\begin{proof}
Case $n = 2$.  Let
$$
u = u_1(v+\alpha_1) = u_1v + \alpha_1u_1,
$$
$$
u = u_2(v+\alpha_2) = u_2v + \alpha_2u_2.
$$
Then 
$$
(u_1-u_2)v+\alpha_1(u_1-u_2) = (\alpha_2-\alpha_1)u_2.
$$
Notice that $\alpha_2-\alpha_1 \neq 0$. It follows that 
$$
(\alpha_2-\alpha_1)u_2 = (u_1-u_2)(v+\alpha_1).
$$
Hence
$$
(\alpha_2-\alpha_1)u = (\alpha_2-\alpha_1)u_2(v+\alpha_2)  = (u_1-u_2)(v+\alpha_1)(v+\alpha_2),
$$
as required.

Case $n \geq 3$. Let

$$u = u_1(v+\alpha_1) , u = u_2(v+\alpha_2), \ldots, u = u_n(v+\alpha_n). $$

By induction from the first $n-1$ equalities one has 
$$
\gamma_1 u = w_1(v+\alpha_1) \ldots (v+\alpha_{n-1}) = w_1' (v+\alpha_{n-1}),
$$
where $0 \neq \gamma_1 \in R$, $w_1' = w_1(v+\alpha_1) \ldots (v+\alpha_{n-2})$. Similarly, considering the system obtained from the initial one above by removing the equality $u = u_{n-1}(v+\alpha_{n-1}) $ one gets by induction that 
$$
\gamma_2 u = w_2(v+\alpha_1) \ldots (v+\alpha_{n-2}) (v+\alpha_n) = w_2' (v+\alpha_n),
$$
where $w_2' = w_2(v+\alpha_1) \ldots (v+\alpha_{n-2})$.

Consider a system
$$\gamma_1 u =w_1' (v+\alpha_{n-1})$$
$$\gamma_2 u =w_2' (v+\alpha_n)$$
Multiplying the first equation by $\gamma_2$ and the second - by $\gamma_1$ one gets

$$\gamma u = \gamma_2 w_1' (v+\alpha_{n-1}),$$
$$\gamma  u = \gamma_1 w_2' (v+\alpha_n),$$
where $\gamma = \gamma_1 \gamma_2 \neq 0$.
From the case $n=2$ one gets
$$
(\alpha_n-\alpha_{n-1}) \gamma  u =  (\gamma_2 w_1'- \gamma_1w_2')(v+\alpha_{n-1})(v+\alpha_n).
$$
Observe, that 
$$
\gamma_2 w_1'- \gamma_1w_2' = \gamma_2 w_1(v+\alpha_1) \ldots (v+\alpha_{n-2}) - \gamma_1w_2(v+\alpha_1) \ldots (v+\alpha_{n-2}) = 
$$
$$
(\gamma_2 w_1 - \gamma_1 w_2)(v+\alpha_1) \ldots (v+\alpha_{n-2}).
$$
Hence
$$
\gamma u  = ((\alpha_n-\alpha_{n-1})^{-1}(\gamma_2w_1 - \gamma_1 w_2))(v+\alpha_1) \ldots (v+\alpha_{n-2})(v+\alpha_{n-1})(v+\alpha_n),
$$
 as claimed.
\end{proof}

\begin{theorem} \label{th:Z} Let $R$ be an integral domain of characteristic 0 and $L$ a free non-commutative Lie algebra over $R$.  Then 
\begin{itemize}
\item [1)] For any  $b \in L$, $b \neq 0$,  the  formula
$$\phi (x,b)=(x\in R) \wedge \exists v\neq 0\exists u\forall k\in R\forall u_1\exists u_2 (v=ub\wedge  (v=u_1(b+k)\implies (v=u_2(b+k+1)\vee k=x)))$$
interprets  $\N\subseteq R$ in $L$ (in the formula above notation $x \in R$, as well as the action of an $\alpha \in R$ on $u \in L$,  means here that $x$ belongs to the interpretation of  $R$  in $L$ from Theorem \ref{th:field} and the action by $\alpha$  is also from this interpretation).
\item [2)] The formula $\exists b [(b \neq 0) \wedge \phi (x,b)]$ 0-interprets $\N \subseteq R$ in $L$.
\end{itemize}
\end{theorem}
\begin{proof} We prove 1) first.  Let $m\in\N$. We need to show that $L \models \phi(m)$. Take any $a \in L$, $a \neq 0$ and put $v=ab(b+1)\ldots (b+m).$ Then,  in the notation above, for any $k\in \mathbb N$ for any $i\leq k$ there is $u_i\in L$ such that $ab(b+1)\ldots (b+k)=u_i (b+i)$. Indeed, by the property b) above for any $w\in L$, and for any $i,j\in R$
$$w(b+i)(b+j)=w(b+j)(b+i).$$ 
This allows one to push $(b+i)$ to the right in the "product" $ab(b+1)\ldots (b+k)$. 

Observe that for any $\alpha \in R \smallsetminus \{0,1, \ldots,m \}$ $v \neq u(b+\alpha)$ for any $u \in L$. Indeed, if $v = u(b+\alpha)$ for such an $\alpha$, then by Lemma \ref{le:main} 
$$
\gamma v= u(b+0)(b+1)\ldots (b+m)(b+\alpha).
$$
In this case by the properties above $wt(v) \geq m+2$, while by the choice of $v$ we have $wt(v) = wt(a(b+0)(b+1)\ldots (b+m)) = m+1$ -  contradiction.

This shows that $\phi(m)$ holds in $L$.

Let now  $x \in  R \smallsetminus  \N$. We need to show that  $L \not \models \phi (x).$  Suppose to the contrary that $L  \models \phi (x)$ for $x \in  R \smallsetminus  \N$.

  Then there exists $v \in L, v \neq 0$ such that 
   $$v= u_0b = u_1(b+1)=u_2(b+2)=\ldots  = u_{n+1}(b+n+1) = \ldots$$
      for some $u_i \in L, i \in \N$.
   
   Then by Lemma \ref{le:main} for any $n \in \N$ one has 
   $$
  \gamma_n v = w_n(b+0)(b+1) \ldots (b+n)
   $$
   for some $0 \neq \gamma_n \in R$ and $w_n \in L$.   Hence, since $w_nb \neq  0$ (otherwise $v = 0$, but  it is not), one has $wt(v) > n$ for every $n \in \N$, but this is impossible since $v \neq 0$.  Hence $L \not \models \phi (x)$, as required.
   
   2) follows immediately from 1). 
   This proves the theorem.
      
  \end{proof}
  
  This result answers the question posed by Baudisch  in  \cite{Bau} in the case of characteristic zero.
 
 \section{Results}

  The following theorem answers questions by  Baudisch  in \cite{Bau} and by Bokut' and  Kukin \cite{BK} in the case of characteristic zero.
\begin{theorem} The first order theory in the ring language of a free non-commutative  Lie algebra over an integral domain of characteristic zero is undecidable.
\end{theorem} 
\begin{proof}
By Theorem \ref{th:Z} the arithmetic $\N$ is interpretable in $L$ in the ring language. Hence the theory $Th(L)$ is undecidable.
\end{proof}
Let $T$ be a complete theory in a language $\mathcal L$. An $\mathcal  L$-formula $\phi(x,y)$ is said to have the independence property (with respect to $x$, $y$) if in every model $M$ of $T$ there is, for each $n = \{0,1,\ldots,n-1\} < \omega$, a family of tuples $b_0,\ldots ,b_{n-1}$  such that for each of the $2^n$ subsets $X$ of $n$ there is a tuple $a \in  M$ for which
$$M\models \varphi (a,b_{i})\quad \Leftrightarrow \quad i\in X.$$
The theory $T$ has independence property if some formula does.

Note that the elementary theory of the arithmetic $\N = \langle\N, +,\cdot,0 \rangle$ is independent. Indeed, the formula "y divides x", i.e., the formula $\exists k (x= ky)$ has the independence property. Clearly the independence property is inherited under interpretations.  The following theorem answers the question posed by Baudisch  in \cite{Bau} in the case of characteristic zero. 
  
\begin{theorem} The first order theory of a free non-commutative  Lie algebra over an integral domain of characteristic zero has the independence property. 
\end{theorem}

\end{document}